\documentclass[12pt]{article}
\usepackage[centertags]{amsmath}
\usepackage{amsfonts}
\usepackage{amssymb}
\usepackage{latexsym}
\usepackage{amsthm}
\usepackage{newlfont}
\usepackage{graphicx}
\usepackage{listings}
\usepackage{booktabs}
\usepackage{abstract}
\date{}

\newlength{\defbaselineskip}
\newcommand{\setlinespacing}[1]%
           {\setlength{\baselineskip}{#1 \defbaselineskip}}

\newcommand{\N}{{\mathbb{N}}}

\newcommand{\actaqed}{\hfill $\actabox$}
{\medskip\noindent \textit{Proof of #1. }}%
{\actaqed \medskip}

\def\cB{\mathcal B}

\def \cF{\mathcal F}

\def\R{{\mathbb R}}
\def\Z{\mathbb Z}

\def \<{\langle}
\def\>{\rangle}

\def\la{\lambda}

\def \supp{\operatorname{supp}}

\def\ba{\mathbf a}
\def\bb{\mathbf b}
\def\bx{\mathbf x}
\def\by{\mathbf y}
\def\bz{\mathbf z}
\def\bk{\mathbf k}
\def\bu{\mathbf u}

\def\bm{\mathbf m}

\def\bs{\mathbf s}

\def\btt{\mathbf t}

\newtheorem{Theorem}{Theorem}[section]
\newtheorem{Lemma}{Lemma}[section]

\newtheorem{Corollary}{Corollary}[section]
\newtheorem{Conjecture}{Conjecture}[section]
\numberwithin{equation}{section}

\newcommand{\be}{\begin{equation}}
\newcommand{\ee}{\end{equation}}

\begin{document}

\title{Fixed volume discrepancy in the periodic case}
\author{V.N. Temlyakov\thanks{University of South Carolina and Steklov Institute of Mathematics.  }}
\maketitle
\begin{abstract}
{The smooth fixed volume discrepancy in the periodic case is studied here. It is proved 
that the Frolov point sets adjusted to the periodic case have optimal in a certain sense 
order of decay of the smooth periodic discrepancy. The upper bounds for the $r$-smooth fixed volume periodic discrepancy for these sets are established. }
\end{abstract}

\section{Introduction} 
\label{I} 

Discrepancy theory is a classical well established area of research in geometry and 
numerical integration (see \cite{BC}, \cite{Mat}, \cite{TBook}, \cite{T11}). Recently, in \cite{VT163}, a new phenomenon has been discovered. A typical upper bound for 
the discrepancy of a good point set of cardinality $m$ is $\le C(d) m^{-1}(\log m)^{d-1}$
and for the $r$-smooth discrepancy $\le C(d,r) m^{-r}(\log m)^{d-1}$. These bounds are 
too rough for functions with small volume of their support. It was proved in \cite{VT163}
that for the Fibonacci point sets ($d=2$) and the Frolov point sets we can improve the above upper bound to $\le C(d,r) m^{-r}(\log mV)^{d-1}$, $V\ge c(r,d)/m$, for the functions with the volume of their support equals $V$. We establish a similar phenomenon 
for the $r$-smooth fixed volume discrepancy in the periodic case. 

We begin with a classical definition of discrepancy ("star discrepancy", $L_\infty$-discrepancy) of a point set $T:=\xi := \{\xi^\mu\}_{\mu=1}^m\subset [0,1)^d$. 
Let $d\ge 2$ and $[0,1)^d$ be the $d$-dimensional unit cube. For $\bx,\by \in [0,1)^d$ with $\bx=(x_1,\dots,x_d)$ and $\by=(y_1,\dots,y_d)$ we write $\bx < \by$ if this inequality holds coordinate-wise. For $\bx<\by$ we write $[\bx,\by)$ for the axis-parallel box $[x_1,y_1)\times\cdots\times[x_d,y_d)$ and define
$$
\cB:= \{[\bx,\by): \bx,\by\in [0,1)^d, \bx<\by\}.
$$

Introduce a class of special $d$-variate characteristic functions
$$
\chi^d := \{\chi_{[\mathbf 0,\bb)}(\bx):=\prod_{j=1}^d \chi_{[0,b_j)}(x_j),\quad b_j\in [0,1),\quad j=1,\dots,d\}
$$
where $\chi_{[a,b)}(x)$ is a univariate characteristic function of the interval $[a,b)$. 
The classical definition of discrepancy of a set $T$ of points $\{\xi^1,\dots,\xi^m\}\subset [0,1)^d$ is as follows
$$
D(T,m,d)_\infty := \max_{\bb\in [0,1)^d}\left|\prod_{j=1}^db_j -\frac{1}{m}\sum_{\mu=1}^m \chi_{[\mathbf 0,\bb)}(\xi^\mu)\right|.
$$
It is equivalent within multiplicative constants, which may only depend on $d$, to the following definition
\be\label{1.1}
D^1(T):=  \sup_{B\in\cB}\left|vol(B)-\frac{1}{m}\sum_{\mu=1}^m \chi_B(\xi^\mu)\right|,
\ee
where for $B=[\ba,\bb)\in \cB$ we denote $\chi_B(\bx):= \prod_{j=1}^d \chi_{[a_j,b_j)}(x_j)$.   Moreover, we consider the following optimized version of $D^1(T)$
\be\label{1.1'}
D^{1,o}(T):= \inf_{\la_1,\dots,\la_m} \sup_{B\in\cB}\left|vol(B)- \sum_{\mu=1}^m \la_\mu \chi_B(\xi^\mu)\right|.
\ee

In the definition of $D^1(T)$ and $D^{1,o}(T)$ -- the $1$-smooth discrepancy -- we use as a building block the univariate characteristic function.   In numerical integration $L_1$-smoothness of a function plays an important role. A characteristic function of an interval has smoothness $1$ in the $L_1$ norm. This is why we call the corresponding discrepancy characteristics the $1$-smooth discrepancy. In the definition of $D^2(T)$,
$D^{2,o}(T)$, $D^2(T,V)$, and $D^{2,o}(T,V)$ (see below and \cite{VT163}) we use the hat function 
$h_{[-u,u)}(x) =|u-x|$ for $|x|\le u$ and $h_{[-u,u)}(x) =0$ for $|x|\ge u$ instead of the characteristic function $\chi_{[-u/2,u/2)}(x)$. Function $h_{[-u,u)}(x)$ has smoothness $2$ in $L_1$. This fact gives the corresponding name. Note that
$$
h_{[-u,u)}(x) = \chi_{[-u/2,u/2)}(x) \ast \chi_{[-u/2,u/2)}(x),
$$
where
$$
f(x)\ast g(x) := \int_\R f(x-y)g(y)dy.
$$
Now, for $r=1,2,3,\dots$ we inductively define
$$
h^1(x,u):= \chi_{[-u/2,u/2)}(x),\qquad h^2(x,u):= h_{[-u,u)}(x),
$$
$$
h^r(x,u) := h^{r-1}(x,u)\ast h^1(x,u),\qquad r=3,4,\dots.
$$
Then $h^r(x,u)$ has smoothness $r$ in $L_1$ and has support $(-ru/2,ru/2)$. 
Represent a box $B\in\cB$ in the form
$$
B= \prod_{j=1}^d [x^0_j-ru_j/2,x^0_j+ru/2)
$$
and define
$$
h^r_B(\bx):= h^r(\bx,\bx^0,\bu):=\prod_{j=1}^d h^r(x_j-x^0_j,u_j).
$$

In \cite{VT163} we modified definitions (\ref{1.1}) and (\ref{1.1'}), replacing the characteristic function $\chi_B$ by a smoother hat function $h^r_B$.  

The $r$-smooth discrepancy is now defined as
\be\label{1.2}
D^r(T):=  \sup_{B\in\cB}\left|\int h^r_B(\bx)d\bx-\frac{1}{m}\sum_{\mu=1}^m h^r_B(\xi^\mu)\right|
\ee
and its optimized version as 
\be\label{1.2'}
D^{r,o}(T):=  \inf_{\la_1,\dots,\la_m}\sup_{B\in\cB}\left|\int h^r_B(\bx)d\bx- \sum_{\mu=1}^m \la_\mu h^r_B(\xi^\mu)\right|.
\ee
Note that the known concept of $r$-discrepancy   (see, for instance, \cite{TBook}, \cite{T11}, and Section \ref{Disc} below) is close to the above concept of $r$-smooth discrepancy. 

Along with $D^r(T)$ and $D^{r,o}(T)$ we consider a more refined quantity -- {\it $r$-smooth fixed volume discrepancy} -- defined as follows (see \cite{VT163})
\be\label{1.3}
D^r(T,V):=  \sup_{B\in\cB:vol(B)=V}\left|\int h^r_B(\bx)d\bx-\frac{1}{m}\sum_{\mu=1}^m h^r_B(\xi^\mu)\right|;
\ee
\be\label{1.3'}
D^{r,o}(T,V):=  \inf_{\la_1,\dots,\la_m}\sup_{B\in\cB:vol(B)=V}\left|\int h^r_B(\bx)d\bx- \sum_{\mu=1}^m \la_\mu h^r_B(\xi^\mu)\right|.
\ee
Clearly,
$$
D^r(T) = \sup_{V\in(0,1]} D^r(T,V).
$$

In Section \ref{Frol} of this paper we study a periodic analog of the quantities $D^{r,o}(T,V)$ for a set $T$ generated with a help of the Frolov lattice. We first describe 
the periodic analogs of the above discrepancy concepts. For a function $f\in L_1(\R^d)$ with a compact support we define its periodization $\tilde f$ as follows
$$
\tilde f(\bx) := \sum_{\bm\in \Z^d} f(\bm+\bx).
$$
Consider $\bu \in (0,\frac{1}{2}]^d$. Then for all $\bz\in [0,1)^d$ we have 
$$
\supp(h^r(\bx,\bz,\bu)) \subset (-r/4,1+r/4)^d.
$$
Now, for each $\bz\in [0,1)^d$ consider a periodization of function $h^r(\bx,\bz,\bu)$ 
in $\bx$ with period $1$ in each variable  
$\tilde h^r(\bx,\bz,\bu)$.  It is convenient for us to use the following abbreviated notation for the product
$$
pr(\bu):= pr(\bu,d) := \prod_{j=1}^d u_j.
$$
Define the corresponding periodic discrepancy as follows
(we only give one modified definition)
$$
\tilde D^{r,o}(T,v):=
$$
\be\label{1.4}
   \inf_{\la_1,\dots,\la_m}\sup_{\bz\in[0,1)^d;\bu:pr(\bu)=v}\left|\int_{[0,1)^d} \tilde h^r(\bx,\bz,\bu)d\bx- \sum_{\mu=1}^m \la_\mu \tilde h^r(\xi^\mu,\bz,\bu)\right|.
\ee

Second we describe the Frolov cubature formulas. We refer the reader for detailed 
presentation of the theory of the Frolov cubature formulas to \cite{TBook}, \cite{T11}, \cite{MUll}, and \cite{DTU}. The following lemma plays a fundamental role in the
construction of such point sets (see \cite{TBook} for its proof).

\begin{Lemma}\label{L1.1} There exists a matrix $A$ such that the lattice
$L(\mathbf m) = A\mathbf m$
$$
 L(\mathbf m) =
\begin{bmatrix}
L_1(\mathbf m)\\
\vdots\\
L_d(\mathbf m)
\end{bmatrix},
$$
where $\mathbf m$ is a (column)
vector with integer coordinates,
has the following properties

{$1^0$}. $\qquad \left |\prod_{j=1}^d L_j(\mathbf m)\right|\ge 1$
for all $\mathbf m \ne \mathbf 0$;

{$2^0$} each parallelepiped $P$ with volume $|P|$
whose edges are parallel
to the coordinate axes contains no more than $|P| + 1$ lattice
points.
\end{Lemma}

Let $a > 1$ and $A$ be the matrix from Lemma \ref{L1.1}. We consider the
cubature formula
$$
\Phi(a,A)(f) := \bigl(a^d |\det A|\bigr)^{-1}\sum_{\mathbf m\in\Z^d}f
\left (\frac{(A^{-1})^T\mathbf m}{a}\right)
$$
for $f$ with compact support.   

We call the {\it Frolov point set} the following set associated with the matrix $A$ and parameter $a$
$$
\cF(a,A) := \left\{\left (\frac{(A^{-1})^T\mathbf m}{a}\right)\right\}_{\bm\in\Z^d}\cap [0,1)^d =: \{\bz^\mu\}_{\mu=1}^N.
$$
 Clearly, the number $N=|\cF(a,A)|$ of points of this
set does not exceed $C(A)a^d$. 
The following results were obtained in \cite{VT163}.

\begin{Theorem}\label{T1.1} Let $r\ge 2$. There exists a constant $c(d,A,r)>0$ such that for any $V\ge V_0:= c(d,A,r)a^{-d}$ we have for all $B\in\cB$, $vol(B)=V$,
\be\label{1.5}
|\Phi(a,A)(h^r_B) - \hat h^r_B(\mathbf 0)| \le C(d,A,r)a^{-rd} (\log(2V/V_0))^{d-1}.
\ee
\end{Theorem}

\begin{Corollary}\label{C1.1} For $r\ge2$ there exists a constant $c(d,A,r)>0$ such that for any $V\ge V_0:= c(d,A,r)a^{-d}$ we have
\be\label{1.6}
D^{r,o}(\cF(a,A),V) \le C(d,A,r)a^{-rd} (\log(2V/V_0))^{d-1}.
\ee
\end{Corollary}

In Section \ref{Frol} we extend Theorem \ref{T1.1} and Corollary \ref{C1.1} to the periodic case. For that we need to modify the set $\cF(a,A)$ and the cubature formula 
$\Phi(a,A)$. For $\by\in \R^d$ denote $\{\by\} := (\{y_1\},\dots,\{y_d\})$, where for $y\in \R$ notation $\{y\}$ means the fractional part of $y$. For given $a$ and $A$ denote
$$
\eta:= \{\eta^\mu\}_{\mu=1}^m := \left\{\left (\frac{(A^{-1})^T\mathbf m}{a}\right)\right\}_{\bm\in\Z^d}\cap [-1/2,3/2)^d
$$
and
\be\label{1.6'}
\xi:=\{\xi^\mu\}_{\mu=1}^m := \left\{\{\eta^\mu\}\right\}_{\mu=1}^m.
\ee
Clearly, $m\le C(A)a^d$. Next, let $w(t)$ be infinitely differentiable on $\R$ function 
with the following properties
\be\label{1.7}
\supp(w) \subset (-1/2,3/2)\quad \text{and}\quad \sum_{k\in\Z} w(t+k) =1.
\ee
Denote $w(\bx):= \prod_{j=1}^d w(x_j)$. Then for $f(\bx)$ defined on $[0,1)^d$ we 
consider the cubature formula
$$
\Phi(a,A,w)(f) := \sum_{\mu=1}^m w_\mu f(\xi^\mu),\qquad w_\mu := w(\eta^\mu) .
$$
 In Section \ref{Frol} we prove the following analogs of Theorem \ref{T1.1} and Corollary \ref{C1.1}.
 
 \begin{Theorem}\label{T1.2} Let $r\ge 2$. There exists a constant $c(d,A,r)>0$ such that for any $v\ge v_0:= c(d,A,r)a^{-d}$ we have for all $\bu\in (0,1/2]^d$, $pr(\bu)=v$, and $\bz\in[0,1)^d$
$$
|\Phi(a,A,w)(\tilde h^r(\cdot,\bz,\bu) - \hat{\tilde h}^r(\mathbf 0,\bz,\bu)| \le C(d,A,r,w)a^{-rd} (\log(2v/v_0))^{d-1}.
$$
\end{Theorem}

\begin{Corollary}\label{C1.2} For $r\ge2$ there exists a constant $c(d,A,r)>0$ such that for any $v\ge v_0:= c(d,A,r)a^{-d}$ we have for the point set $\xi$ defined by (\ref{1.6'})
$$
\tilde D^{r,o}(\xi,v) \le C(d,A,r)a^{-rd} (\log(2v/v_0))^{d-1}.
$$
\end{Corollary}

In particular, Theorem \ref{T1.2} implies that the $r$-smooth periodic discrepancy
$$
\tilde D^{r,o}_m :=
$$
\be\label{1.4'}
   \inf_{\la_1,\dots,\la_m}\sup_{\bz\in[0,1)^d;\bu\in(0,1/2]^d}\left|\int_{[0,1)^d} \tilde h^r(\bx,\bz,\bu)d\bx- \sum_{\mu=1}^m \la_\mu \tilde h^r(\xi^\mu,\bz,\bu)\right|
\ee
satisfies the bound (for $r\in\N$, $r\ge 2$)
\be\label{1.13}
\tilde D^{r,o}_m   \le C(d,r)m^{-r} (\log m)^{d-1}.
\ee
In Section \ref{lb} we show that the bound (\ref{1.13}) cannot be improved for a natural 
class of weights $\la_1,\dots,\la_m$ used in the optimization procedure in the definition 
of $\tilde D^{r,o}_m$, namely, for weights, satisfying
$$
\sum_{\mu=1}^m |\la_\mu| \le B.
$$

\section{Point sets based on the Frolov lattice}
\label{Frol}
We prove Theorem \ref{T1.2} in this section.
Let $f(\bx)$ be $1$-periodic in each variable function integrable on $\Omega_d:= [0,1)^d$. Then function $w(\bx)f(\bx)$ is integrable on $\R^d$ and has a finite support:
 $\supp(wf)\subset (-1/2 ,3/2)^d$. We note that the idea of applying the Frolov cubature 
 formulas to the product of the form $w(\bx)f(\bx)$, where one function is very smooth and takes care of the support of the product (in our case it is $w(\bx)$) and the other function has a prescribed decay of its Fourier coefficients (in our case it is $f(\bx) = h^r(\bx,\bz,\bu)$), goes back to the very first paper \cite{Fro1} on the Frolov cubature formulas. Further detailed development of this idea was made in \cite{NUU}.
 
 Property (\ref{1.7}) implies
$$
\int_{\R^d}w(\bx)f(\bx)d\bx = \sum_{\bk\in\Z^d}\int_{\Omega_d} w(\bk+\bx)f(\bk+\bx)d\bx= \sum_{\bk\in\Z^d}\int_{\Omega_d} w(\bk+\bx)f(\bx)d\bx 
$$
\be\label{2.1}
  =\int_{\Omega_d}\left(\sum_{\bk\in\Z^d}w(\bk+\bx)\right)f(\bx)d\bx =\int_{\Omega_d}f(\bx)d\bx.
\ee
Next, using periodicity of $f$ we write
\be\label{2.2}
\Phi(a,A,w)(f) := \sum_{\mu=1}^m w_\mu f(\xi^\mu) = \sum_{\mu=1}^m w(\eta^\mu) f(\eta^\mu)=\Phi(a,A)(wf).
\ee
Thus, for a $1$-periodic function $f$ we have
\be\label{2.3}
\int_{\Omega_d}f(\bx)d\bx - \Phi(a,A,w)(f) = \int_{\R^d}w(\bx)f(\bx)d\bx - \Phi(a,A)(wf).
\ee
We use identity (\ref{2.3}) for $f(\bx) = \tilde h^r(\bx,\bz,\bu)$ and estimate the right hand side of (\ref{2.3}). It is clear that it is sufficient to estimate 
\be\label{2.4}
\int_{\R^d}w(\bx)h^r(\bx,\bz,\bu)d\bx - \Phi(a,A)(wh^r).
\ee
In the case $w(\bx)=1$ the above error is bounded in \cite{VT163}. We follow a similar way and use some technical lemmas from \cite{VT163}. Denote for $f\in L_1 (\R^d)$
$$
\cF(f)(\by):=\hat f(\mathbf y) := \int_{\R^d} f(\mathbf x)e^{-2\pi i(\mathbf y,\mathbf x)}d\mathbf x.
$$
For a function $f$ with finite support and absolutely convergent 
series \newline $\sum_{\mathbf m\in\Z^d}\hat f(aA\mathbf m)$ we have for the error of the Frolov  cubature formula (see \cite{TBook})
\be\label{2.5}
 \Phi(a,A)(f) -\hat f(\mathbf 0) =\sum_{\mathbf m\ne\mathbf 0}
\hat f(aA\mathbf m).
\ee
We begin with the following simple univariate lemma. 
\begin{Lemma}\label{L2.1} Suppose that $r\in\N$ and $f\in L_1(\R)$ satisfies the following conditions
$$
|\supp(f)|\le C_1u, \quad |f(x)| \le C_2u^{r-1},\quad \|\Delta^{r}_t f\|_1\le C_3|t|^r,
$$
where $\Delta_t f(x):= f(x)-f(x+t)$, $\Delta^r_t:= (\Delta_t)^r$. 
Then,
$$
|\hat f(y)| \le C_4\min\left(u^r,\frac{1}{|y|^r}\right).
$$
\end{Lemma}
\begin{proof} It is easy to see that
$$
{\mathcal F}(f)(y):=\hat f(y)   ={\mathcal F}\left(\frac 12
\Delta_{\frac{1}{2y}}f \right) (y).
$$
Iterating the above identity $r$ times we obtain
$$
{\mathcal F}(f)(y) = {\mathcal F}\left(\left( \frac{1}{2^{r}}
\Delta_{\frac{1}{2y}}^{r}\right)f\right)(y).
$$
Using the above representation and our assumptions on $f$, we get
$$
\bigl|{\mathcal F}(f)(\mathbf y)\bigr|\le C_4\min\left(u^r,\frac{1}{|y|^r}\right).
$$

The lemma is proved.

\end{proof}

We return back to estimation of (\ref{2.4}). We have
$$
w(\bx)h^r(\bx,\bz,\bu)=\prod_{j=1}^d w(x_j)h^r(x_j,z_j,u_j).
$$
It is easy to check that $f(x):= w(x)h^r(x,z,u)$ satisfies conditions of Lemma \ref{L2.1}.
Therefore, for $f(\bx):= w(\bx)h^r(\bx,\bz,\bu)$ by Lemma \ref{L2.1} we have 
$$
|\hat f(\by)| \le C(d,r,w)\prod_{j=1}^d\min\left(u_j^r,\frac{1}{|y_j|^r}\right)
$$
\be\label{2.6}
 =C(d,r,w)\prod_{j=1}^d\left(\frac{u_j}{|y_j|}\right)^{r/2}\min\left(|y_ju_j|^{r/2},\frac{1}{|y_ju_j|^{r/2}}\right). 
\ee
Consider
$$
\sigma^r(n,\bu):= \sum_{\|\bs\|_1=n}\prod_{j=1}^d \min\left((2^{s_j}u_j)^{r/2},\frac{1}{(2^{s_j}u_j)^{r/2}}\right),\quad v\in\N_0.
$$
The following lemma was established in \cite{VT163}. 
\begin{Lemma}\label{L2.2} Let $n\in \N_0$ and $\bu\in (0,1/2]^d$. Then we have 
the following inequalities.

(I) Under condition $2^n pr(\bu)\ge 1$ we have
\be\label{2.7}
\sigma^r(n,\bu) \le C(d) \frac{\left(\log(2^{n+1}pr(\bu))\right)^{d-1}}{(2^n pr(\bu))^{r/2}}.
\ee

(II) Under condition $2^n pr(\bu)\le 1$ we have
\be\label{2.8}
\sigma^r(n,\bu) \le C(d) (2^n pr(\bu))^{r/2} \left(\log\frac{2}{2^{n}pr(\bu)}\right)^{d-1}.
\ee

\end{Lemma}

For $\bs\in \N_0^d$ -- the set of vectors with nonnegative integer coordinates, define
$$
\rho (\bs) := \{\bk \in \Z^d : [2^{s_j-1}] \le |k_j| < 2^{s_j}, \quad j=1,\dots,d\}
$$
where $[a]$ denotes the integer part of $a$. 

By (\ref{2.5}) we have for the error  
$$
\delta := \left|\int f(\bx)d\bx-\Phi(a,A)(f)\right| \le \sum_{n=1}^\infty \sum_{\|\bs\|_1=n}\sum_{\bm: aA\bm\in \rho(\bs)} |\hat f(aA\bm)|.
$$
Lemma \ref{L1.1} implies that if $n\neq 0$ is such that $2^n< a^d$ then for
$\bs$ with $\|\bs\|_1=n$ there is no $\bm$ such that  $aA\bm\in\rho(\bs)$. Let $n_0\in \N$ be the smallest number satisfying 
$2^{n_0}\ge a^d$. Then we have 
\be\label{2.9}
\delta \le \sum_{n=n_0}^\infty \sum_{\|\bs\|_1=n}\sum_{\bm: aA\bm\in \rho(\bs)} |\hat f(aA\bm)|.
\ee
Lemma \ref{L1.1} implies that for $n\ge n_0$ we have
\be\label{2.10}
|\rho(\bs)\cap \{aA\bm\}_{\bm\in \Z^d}| \le C_12^{n-n_0}, \quad \|\bs\|_1=n.
\ee
Using (\ref{2.10}) we obtain by (\ref{2.6}) for $f(\bx) = w(\bx)h^r(\bx,\bz,\bu)$
$$
\delta \le C(d,r,w) \sum_{n=n_0}^\infty \sum_{\|\bs\|_1=n} 2^{n-n_0}(pr(\bu)2^{-n})^{r/2}\prod_{j=1}^d\min\left((2^{s_j}u_j)^{r/2},\frac{1}{(2^{s_j}u_j)^{r/2}}\right).
$$
We now assume that the constant $c(d,A)$ is such that  $v_0= 2^{-n_0}$. Then for $v\ge v_0$ we have $2^n pr(\bu)\ge 1$, $n\ge n_0$. 
Using inequality (\ref{2.7}) of Lemma \ref{L2.2}   we obtain from here
$$
 \delta \le C(d,r,w) 2^{-n_0}\sum_{n=n_0}^\infty 2^{-n(r-1)}\left(\log\left(2^{n+1}pr(\bu)\right)\right)^{d-1} 
 $$
 $$
 \le C(d,r,w) 2^{-rn_0}\left(\log\left(2v/v_0\right)\right)^{d-1} \le C(d,r,w) a^{-rd}\left(\log\left(2v/v_0\right)\right)^{d-1}.
$$

\section{A lower bound for the smooth periodic discrepancy}
\label{lb}

In this section we prove a lower bound for an analog of the smooth periodic discrepancy $\tilde D^{r,o}(T)$ for any set $T$ of fixed cardinality. In fact we prove 
a weaker result. In the definition of optimal smooth periodic discrepancy
\be\label{3.1}
\tilde D^{r,o}(T):=  \inf_{\la_1,\dots,\la_m}\sup_{\bz\in[0,1)^d;\bu\in(0,1/2]^d}\left|\int_{[0,1)^d} \tilde h^r(\bx,\bz,\bu)d\bx- \sum_{\mu=1}^m \la_\mu \tilde h^r(\xi^\mu,\bz,\bu)\right|
\ee
we allow to optimize over all weights $\la_1,\dots,\la_m$. We prove a lower bound under an extra (albeit mild) restriction on the weights. Let $B$ be a positive number and $Q(B,m)$ be the set of cubature formulas 
$$
\Lambda_m(f,\xi) := \sum_{\mu=1}^m \la_\mu f(\xi^\mu),\quad \xi:= \{\xi^\mu\}_{\mu=1}^m \subset [0,1)^d,\quad \la_\mu \in \R,\quad  \mu=1,\dots,m,
$$
 satisfying the additional condition
\be\label{3.1'}
\sum_{\mu=1}^m |\lambda_\mu| \le B.
\ee
We obtain the lower estimates for the quantities
$$
     \tilde D_m^{r,B} := \inf_{\Lambda_m(\cdot,\xi) \in Q(B,m)}\sup_{\bz\in[0,1)^d;\bu\in (0,1/2]^d}\left|\int_{[0,1)^d} \tilde h^r(\bx,\bz,\bu)d\bx- \Lambda( \tilde h^r(\cdot,\bz,\bu),\xi)\right|  .
$$
 We prove the following relation.   
 \begin{Theorem}\label{T3.1} Let $r\in\N$ be an even number. Then
$$
    \tilde D_m^{r,B} \geq   C(r,B,d)   m^{-r}(\log   m)^{d-1},   \qquad
C(r,B,d)>0.
$$
\end{Theorem}
 
\begin{proof} Theorem \ref{T3.1} is an analog of Theorem 3 from \cite{VT50} (see also \cite{T11}). Our proof follows the ideas from \cite{VT50}.
We use a  notation
$$
     \Lambda(\bk) := \Lambda_m(e^{i2\pi(\bk,\bx)},\xi) = \sum_{\mu=1}^m  \lambda_\mu
e^{i2\pi(\bk,\xi^\mu)}.
$$
Let a set $T$ with cardinality $|T|=m$ be given. We specify $\xi:=T$ and consider 
along with the cubature formula $\Lambda(\cdot,\xi)$ the following auxhilary cubature 
formula 
$$
     \Lambda^*(f) :=   \sum_{\nu=1}^m     \lambda_\nu  \Lambda_m(f(\bx  -
\xi^\nu),\xi).
$$
Then
\be\label{3.2}
     \Lambda^*(\bk) =  \Lambda^*(e^{i2\pi(\bk,\bx)}) = \sum_{\nu=1}^m  \lambda_\nu
\sum_{\mu=1}^m \lambda_\mu e^{i2\pi(\bk,(\xi^\mu - \xi^\nu))} = 
      |\Lambda(\bk)|^2. 
\ee
Suppose that for each $f(\bx):=f(\bx,\bu):=\tilde h^r(\bx,\bz,\bu)$ we have for all $\bz\in [0,1)^d$ 
and $\bu\in (0,\frac{1}{2}]^d$ the bound
\be\label{3.3}
|\hat f(\mathbf 0,\bu) - \Lambda_m(f,\xi)| \le e_m.
\ee
Integrating over $[0,1)^d$ with respect to $\bz$ we get from here
\be\label{3.3'}
|\hat f(\mathbf 0,\bu)(1 - \Lambda_m(\mathbf 0))| \le e_m.
\ee
In particular, this implies
\be\label{3.3"}
| 1 - \Lambda_m(\mathbf 0)| \le c(r,d)e_m.
\ee
Therefore,   we have
$$
|\Lambda^*(f) - \Lambda^*(\mathbf 0) \hat f(\mathbf 0)| \le |\Lambda^*(f) - \Lambda_m(\mathbf 0) \hat f(\mathbf 0)| + |(\Lambda_m(\mathbf 0)-\Lambda_m(\mathbf 0)^2) \hat f(\mathbf 0)|
$$
\be\label{3.5}
  \le \left|\sum_{\mu=1}^m \la_\nu\left(\Lambda_m(f(\bx-\xi^\nu),\xi) - \hat f(\mathbf 0)\right)\right| + Be_m\le 2Be_m.
\ee
We now need a known result on the lower bound for the weighted sum of 
$\{|\Lambda(\bk)|^2\}$ (see \cite{TBook} and \cite{T11}).

\begin{Lemma}\label{L3.1} The following inequality is valid for any
$r > 1$
$$
\sum_{ \bk\ne 0}\bigl|\Lambda( \bk)\bigr|^2\nu
(\bar{ \bk})^{-r}\ge
C(r,d)\bigl|\Lambda(\mathbf 0)\bigr|^2 m^{-r}(\log m)^{d-1},
$$
where $\bar{k_j}:=\max(|k_j|,1)$ and $\nu(\bar{\bk}):=\prod_{j=1}^d \bar{k_j}$.
\end{Lemma}
 By (\ref{3.2})   we get
\be\label{3.8}
|\Lambda^*(f) - \Lambda^*(\mathbf 0) \hat f(\mathbf 0)| = \left|\sum_{\bk\neq \mathbf 0} \Lambda^*(\bk) \hat f(\bk)\right| = \left|\sum_{\bk\neq \mathbf 0} |\Lambda(\bk)|^2 \hat f(\bk)\right|.
\ee
Applying (\ref{3.8}) for $f(\bx)= \tilde h^r(\bx,\mathbf 0, \bu)$ and using (\ref{3.5}) we obtain
\be\label{3.9}
2Be_m \ge |\Lambda^*(f) - |\Lambda(\mathbf 0)|^2 \hat f(\mathbf 0)| = \left|\sum_{\bk\neq \mathbf 0} |\Lambda(\bk)|^2  \hat{\tilde h}^r(\bk,\mathbf 0, \bu)\right|.
\ee
Next
$$
\hat {\tilde h}^r(0, 0, u)=u^r,\quad \hat {\tilde h}^r(k, 0, u)=\left(\frac{\sin(\pi ku)}{\pi k}\right)^r,\quad k\neq 0,
$$
which implies for $r$ even
$$
\int_0^{1/2}\hat {\tilde h}^r(k, 0, u)du \ge c(r)(\bar k)^{-r}.
$$
Integrating the right hand side of (\ref{3.9}) with respect to $\bu$ over $(0,1/2]^d$ and using Lemma \ref{L3.1}  we get for $r$ even
$$
e_m \ge C(r,d)|\Lambda(\mathbf 0)|^2 m^{-r} (\log m)^{d-1}. 
$$
It is clear (see, for instance (\ref{3.3"})) that it must be $|\Lambda(\mathbf 0)|\ge c(r,d)>0$. 
This completes the proof of Theorem \ref{T3.1}.

\end{proof}

\section{Discussion}
\label{Disc}

The paper addresses some issues of discrepancy theory. Discrepancy theory is a well established topic with deep elaborate technique and with some open fundamental problems (see, for instance,  \cite{BC}, \cite{Mat}, \cite{TBook}, \cite{T11}). 
One of the most acute open problems is the problem of the right order of decay of 
the quantity
$$
D(m,d)_\infty := \inf_{T} D(T,m,d)_\infty.
$$
The upper bound is known (see \cite{BC})
\be\label{4.1}
D(m,d)_\infty \le C(d) m^{-1} (\log m)^{d-1}.
\ee
In case $d=2$ it is complemented by the lower bound proved by W. Schmidt \cite{Sch1} 
\be\label{4.2}
D(m,2)_\infty \ge Cm^{-1}\log m . 
\ee
In the case $d\ge 3$ the problem is still open. 
The following conjecture has been formulated in \cite{BC} as an excruciatingly difficult great open problem.
\begin{Conjecture}\label{Con4.1} We have for $d\ge 3$
$$
D(m,d)_\infty \ge C(d)m^{-1}(\log m)^{d-1}. 
$$
\end{Conjecture}
This problem is still open. Recently, D. Bilyk and M. Lacey \cite{BL} and D. Bilyk, M. Lacey, and A. Vagharshakyan \cite{BLV} proved
$$
D(m,d)_\infty \ge C(d)m^{-1}(\log m)^{(d-1)/2 + \delta(d)} 
$$
with some positive $\delta(d)$. 

In this paper we introduce a concept of $r$-smooth discrepancy and prove the lower bound for the $r$-smooth periodic discrepancy $\tilde D^{r,B}_m$ (see Theorem \ref{T3.1}) for $r$ even numbers. This lower bound does not prove Conjecture \ref{Con4.1} but it supports it. There is another variant of smooth discrepancy, which 
shows similar behavior. We discuss it in detail (see \cite{TBook} and \cite{T11}). 
In the definition of the $r$-discrepancy instead of the characteristic function (this corresponds to $1$-discrepancy) we use the following function
\begin{align*}
B_r( \btt, \bx)&:= \prod_{j=1}^d\bigl((r-1)!\bigr)^{-1}
(t_j - x_j )_+^{r-1},\\
 \btt, \bx&\in\Omega_d,\qquad (a)_+ := \max (a,0).
\end{align*}
Then for point set $\xi:=\{\xi^\mu\}_{\mu=1}^m$ of cardinality $m$ and weights $\Lambda:=\{\la_\mu\}_{\mu=1}^m$ we define the $r$-discrepancy of the pair $(\xi,\Lambda)$ by the formula
\be\label{4.3}
D_r (\xi,\Lambda,m,d)_{\infty}:=\sup_{\btt \in (0,1]^d}\left |\sum_{\mu=1}^{m}\lambda_{\mu}B_r ( \btt,\xi^{\mu})-
\prod_{j=1}^d (t_j^r /r!)\right|  .
\ee
Define
$$
D_r (\xi,m,d)_{\infty}:= \inf_{\Lambda}D_r (\xi,\Lambda,m,d)_{\infty}.
$$
Then $D_r (\xi,m,d)_{\infty}$ is close in a spirit to the quantity $D^{r,o}(\xi)$ defined in 
(\ref{1.2'}). The following known result (see \cite{TBook} and \cite{T11}) gives the lower bounds in the case of weights $\Lambda$ satisfying an extra condition (\ref{3.1'}).
\begin{Theorem}\label{T4.1} Let $B$ be a positive number. For any points $\xi^1,\dots,\xi^m \subset \Omega_d$ and any weights $\Lambda =(\lambda_1,\dots,\lambda_m)$ satisfying the condition
$$
\sum_{\mu=1}^m|\lambda_\mu| \le B
$$
we have for even integers $r$
$$
D_r(\xi,\Lambda,m,d)_\infty \ge C(d,B,r)m^{-r}(\log m)^{d-1}
$$
with a positive constant $C(d,B,r)$.
\end{Theorem}

Theorem \ref{T3.1} is an analog of the above Theorem \ref{T4.1}.

The concept of fixed volume discrepancy was introduced and studied in \cite{VT163}.
It is an interesting concept by itself and it is closely related to the concept of dispersion. 
For $n\ge 1$ let $T$ be a set of points in $[0,1)^d$ of cardinality $|T|=n$. The volume of the largest empty (from points of $T$) axis-parallel box, which can be inscribed in $[0,1)^d$, is called the dispersion of $T$:
$$
\text{disp}(T) := \sup_{B\in\cB: B\cap T =\emptyset} vol(B).
$$
An interesting extremal problem is to find (estimate) the minimal dispersion of point sets of fixed cardinality:
$$
\text{disp*}(n,d) := \inf_{T\subset [0,1)^d, |T|=n} \text{disp}(T).
$$
It is known that 
\be\label{4.4}
\text{disp*}(n,d) \le C^*(d)/n.
\ee 
A trivial lower bound disp*$(n,d) \ge (n+1)^{-1}$ combined with (\ref{4.4}) shows that the optimal rate of decay 
of dispersion with respect to cardinality $n$ of sets is $1/n$. 
Another interesting problem is to find (provide a construction) of sets $T$ with cardinality $n$, which have optimal rate of decay of dispersion: $\text{disp}(T)\le C(d)/n$.
Inequality (\ref{4.4}) with $C^*(d)=2^{d-1}\prod_{i=1}^{d-1}p_i$, where $p_i$ denotes the $i$th prime number, was proved in \cite{DJ} (see also \cite{RT}). The authors of \cite{DJ} used the Halton-Hammersly set of $n$ points (see \cite{Mat}). Inequality (\ref{4.4}) with $C^*(d)=2^{7d+1}$ was proved in 
\cite{AHR}. The authors of \cite{AHR}, following G. Larcher, used the $(t,r,d)$-nets (see \cite{NX} and \cite{Mat} for results on $(t,r,d)$-nets).   For further recent results on dispersion we refer the reader to papers \cite{Ull}, \cite{Rud}, \cite{Sos} and references therein.
In \cite{VT163} we proved that 
the Fibonacci and the Frolov point sets have optimal in the sense of order rate of decay 
of dispersion. This result was derived from the bounds on the $2$-smooth fixed volume discrepancy of the corresponding point sets. In the case of the Frolov point sets it is provided by Corollary \ref{C1.1} formulated above.

{\bf Acknowledgment.} The author would like to thank the Erwin Schr{\"o}dinger International Institute
for Mathematics and Physics (ESI) at the
University of Vienna for support. This paper was written, when the author participated in 
the ESI-Semester "Tractability of High Dimensional Problems and Discrepancy",
September 11--October 13, 2017.

\end{document}